\documentclass[11pt]{article}

\usepackage{amsmath, amsfonts, amssymb, amsthm, mathrsfs, mathtools}
\usepackage{graphicx}
\usepackage{color}
\usepackage{epsf}
\usepackage{tikz}
\usepackage{bbm, bm}
\usepackage{subcaption}
\usepackage{hyperref}
\usepackage{arydshln}
\hypersetup{
    colorlinks = true,
    linkcolor = {black},
    citecolor = {black}
}

\usepackage[left=1.3in, right=1.3in, top=1.3in, bottom=1.3in]{geometry}

\let\OLDthebibliography\thebibliography
\renewcommand\thebibliography[1]{
  \OLDthebibliography{#1}
  \setlength{\parskip}{0pt}
  \setlength{\itemsep}{1pt plus 1ex}
}

\usepackage{enumitem}
\setlist{itemsep=0pt, topsep=2pt}

\usepackage{crossreftools}

\makeatletter
\newcommand{\optionaldesc}[2]{%
  \phantomsection
  #1\protected@edef\@currentlabel{#1}\label{#2}%
}
\makeatother

\usepackage{tikz}
\usetikzlibrary{fadings,calc}

\DeclareMathOperator{\ran}{ran}
\DeclareMathOperator{\dom}{dom}
\newcommand{\cof}{\operatorname{cof}}

\newcommand{\NN}{\mathbb{N}}

\newcommand{\lex}{\mathrm{lex}}

\newtheorem{theorem}{Theorem}[section]
\newtheorem{corollary}[theorem]{Corollary}
\newtheorem{lemma}[theorem]{Lemma}

\newtheorem{proposition}[theorem]{Proposition}
\newtheorem{observation}[theorem]{Observation}

\newtheorem{question}[theorem]{Question}

\newtheorem{remark}[theorem]{Remark}

\newcounter{capitalcounter}

\begin{document}
\title{Unavoidable structures in infinite tournaments}
\author{Alistair Benford\thanks{School of Mathematics, University of Birmingham, Birmingham, B15 2TT, UK. \texttt{a.s.benford@bham.ac.uk}} \and Louis DeBiasio\thanks{Department of Mathematics, Miami University, Oxford, OH. \texttt{debiasld@miamioh.edu}. Research supported in part by NSF grant DMS-1954170.} \and Paul Larson\thanks{Department of Mathematics, Miami University, Oxford, OH. \texttt{larsonpb@miamioh.edu}.}}

\maketitle

\begin{abstract}
We prove a strong dichotomy result for countably-infinite oriented graphs; that is, we prove that for all countably-infinite oriented graphs $G$, either (i) there is a countably-infinite tournament $K$ such that $G\not\subseteq K$, or (ii) every countably-infinite tournament contains a \emph{spanning} copy of $G$.  Furthermore, we are able to give a concise characterization of such oriented graphs.  Our characterization becomes even simpler in the case of transitive acyclic oriented graphs (i.e.~strict partial orders). 

For uncountable oriented graphs, we are able to extend the dichotomy result mentioned above to all regular cardinals $\kappa$; however, we are only able to provide a concise characterization in the case when $\kappa=\aleph_1$.  
\end{abstract}

\section{Introduction}

An oriented graph $G$ is an anti-symmetric directed graph (that is, if $(u,v)\in E(G)$, then $(v,u)\not\in E(G)$).  For other standard terminology, see Section \ref{sec:not}.

A finite oriented graph is \emph{unavoidable} if there exists a positive integer $N$ such that $G$ is subgraph of every tournament on at least $N$ vertices.  In the case of finite oriented graphs $G$, it is easy to see that $G$ is unavoidable if and only if $G$ is acyclic.  Indeed, an oriented graph $G$ on $n$ vertices is acyclic if and only if $G$ is a subgraph of the transitive tournament on $n$ vertices, and it is well-known (see \cite{EM}) that every tournament on at least $2^{n-1}$ vertices contains a transitive tournament of order $n$.  This leads to the following natural definition.  For a finite acyclic oriented graph $G$, let $\vec{r}(G)$ be the smallest integer $N$ such that every tournament on $N$ vertices contains a copy of $G$.  While there were some earlier sporadic results, the systematic study of bounding $\vec{r}(G)$ for general $G$ began with \cite{LSS} and \cite{SS}.  A few of the major results in this area are as follows: for the transitive tournament $\vec{K}_n$ on $n$ vertices, $2^{n/2}\leq \vec{r}(\vec{K}_n)\leq 2^{n-1}$ \cite{EM}, for every oriented path $P$ on $n\geq 8$ vertices, $\vec{r}(P)=n$ \cite{HT}, and for every oriented tree $T$ on $n$ vertices (for sufficiently large $n$), $\vec{r}(T)\leq 2n-2$, which is best possible for certain oriented trees $T$ \cite{KMO} (for a refinement of this result, see \cite{DH}, \cite{BM1}, \cite{BM2}).

The goal of this paper is to extend the notion of unavoidability to infinite oriented graphs.  Given an infinite cardinal $\kappa$, we say that an oriented graph $G$ is \emph{$\kappa$-unavoidable} if $G$ is a subgraph of every tournament of cardinality $\kappa$.  When $\kappa=\omega$ (i.e.~when $\kappa$ is countably-infinite), we simply say $G$ is \emph{unavoidable}.  It turns out that in the countably-infinite case, it is not true that an oriented graph $G$ is unavoidable if and only if $G$ is acyclic.  So our first goal is to characterize which countably-infinite oriented graphs are unavoidable.  Having done that, the next goal would be to get quantitative results for unavoidable oriented graphs along the lines of the results mentioned above.  For instance, motivated by recent Ramsey-type results regarding monochromatic subgraphs in edge-colorings of $K_{\NN}$ (\cite{DM}, \cite{CDM}, \cite{CDLL}, \cite{L}, \cite{BL}), it would be natural try to prove that there exists $d>0$ such that for every countably-infinite unavoidable oriented tree $T$ and every tournament $K$ on $\NN$, there is an embedding $\varphi:T\to K$ such that $\varphi(V(T))\subseteq \NN$ has upper density at least $d$.

So it is perhaps surprising that we prove the following result which both characterizes unavoidable oriented graphs and proves that all such countably-infinite unavoidable oriented graphs are unavoidable in a very strong sense (in a way which makes the quantitative question mentioned above irrelevant).  We say that an countably-infinite oriented graph $G$ is \emph{strongly unavoidable} if $G$ is a spanning subgraph of every countably-infinite tournament.

\begin{theorem}\label{thm:main}
Let $G$ be an countably-infinite oriented graph.  The following are equivalent:
\begin{enumerate}[label = {(\bfseries{C\arabic{enumi}})}]
\item\label{c1} $G$ is strongly unavoidable.
\item\label{c2} $G$ is unavoidable.
\item\label{c3} $G\subseteq K_{\omega}$ and $G\subseteq K_{\omega^*}$.
\item\label{c4} $G$ is acyclic, locally-finite, and has no infinite directed paths. 
\end{enumerate}
\end{theorem}

Note that trivially, \ref{c1}$\Rightarrow$\ref{c2}$\Rightarrow$\ref{c3} and is it not hard to see that \ref{c3}$\Rightarrow$\ref{c4}.  As we will see, Ramsey's theorem implies \ref{c2}$\Leftrightarrow$\ref{c3} and another classic order-theoretic result implies that \ref{c3}$\Leftrightarrow$\ref{c4}.  The surprising part, and the main result of the paper, is that \ref{c4}$\Rightarrow$\ref{c1}.

A strict partial order $P=(V,\prec)$ is a relation $\prec$ on a set $V$ which is anti-reflexive, anti-symmetric, and transitive.  Defined in this way, every strict partial order $P=(V,\prec)$ is an acyclic oriented graph.  While it is not the case that every acyclic oriented graph is a strict partial order, there is an equivalence relation on acyclic oriented graphs, where $G\sim F$ if and only if the transitive closures of $G$ and $F$ are isomorphic, where the equivalence classes correspond to strict partial orders.  Note that if $G$ is acyclic, locally-finite, and has no infinite directed paths, then the transitive closure of $G$ is a strict partial order in which every element is comparable to finitely many others (note that we aren't using the phrase ``locally-finite'' in the context of partial orders since this seems to have a different meaning in the literature) and every strict partial order in which every element is comparable to finitely many others is acyclic, locally-finite, and has no infinite directed paths. So we have the following corollary of Theorem \ref{thm:main}.

\begin{corollary}\label{cor:poset}
Let $P=(V,\prec)$ be an countably-infinite strict partial order.  Then $P$ is strongly unavoidable if and only if $P$ is unavoidable if and only if every element in $P$ is comparable to finitely many others.
\end{corollary}

For the uncountable case, the situation is more complicated and will be discussed in much greater detail in Section \ref{sec:uncount}.  However, the main results from Section \ref{sec:uncount} are simple enough to state.  The first says that for regular cardinals, the statements analogous to \ref{c1} and \ref{c2} are equivalent and the statements analogous to \ref{c3} and \ref{c4} are equivalent (and trivially the statement analogous to \ref{c2} implies the statement analogous to \ref{c3}).  

\begin{theorem}\label{thm:main2}
Let $\kappa$ be a uncountable regular cardinal and let $G$ be an oriented graph with $|V(G)|=\kappa$.  
\begin{enumerate}
\item $G$ is strongly $\kappa$-unavoidable if and only if $G$ is $\kappa$-unavoidable.
\item $G\subseteq K_{\kappa}$ and $G\subseteq K_{\kappa^*}$ if and only if $G$ is acyclic, has no infinite directed paths, and every vertex has degree less than $\kappa$.
\end{enumerate}
\end{theorem}

Another main result from Section \ref{sec:uncount} is that for $\kappa=\aleph_1$, the analogue of Theorem \ref{thm:main} holds; that is, the statement analogous to \ref{c4} implies the statement analogous to \ref{c1}.

\begin{theorem}\label{thm:main3}
Let $G$ be an oriented graph with $|V(G)|=\aleph_1$.  If $G$ is acyclic, has no infinite directed paths, and every vertex has degree less than $\aleph_1$, then $G$ is strongly $\aleph_1$-unavoidable.
\end{theorem}

On the other hand when $\kappa=\aleph_2$, the analogue of Theorem \ref{thm:main3} does not hold if the Continuum Hypothesis (CH) fails; that is, if $2^{\aleph_{0}} \geq \aleph_{2}$. We do not know if CH is needed for this result, or whether the Generalized Continuum Hypothesis (GCH) implies the corresponding version of Theorem \ref{thm:main3} for cardinals greater than $\aleph_{1}$. 

\subsection{Notation}\label{sec:not}

An (oriented) graph is \emph{locally-finite} if every vertex is incident with finitely many edges. A \emph{directed cycle} on $n$ vertices is the oriented graph $\vec{C}_n$ with $V(\vec{C}_n)=\{x_1, \dots, x_n\}$ and $E(\vec{C}_n)=\{(x_1,x_2), \dots, (x_{n-1}, x_n), (x_n, x_1)\}$. Say that an oriented graph is \emph{acyclic} if it contains no directed cycles.  A \emph{directed path} is an orientation of a finite, one-way-infinite, or two-way-infinite path having the property that there are no vertices of out-degree 2 or in-degree 2.  The infinite directed path with exactly one vertex of in-degree 0 is called the \emph{infinite forward directed path} and the infinite directed path with exactly one vertex of out-degree 0 is called the \emph{infinite backward directed path}.

Say that an oriented graph $G$ is \emph{connected} if the underlying graph (i.e. the symmetric closure of $G$) is connected.  Note that this is typically referred to as ``weak connectivity'', but since we are only considering acyclic graphs $G$, the notion of ``strong connectivity'' will not arise.

Given oriented graphs $H$ and $G$, an \emph{embedding} of $H$ into $G$, denoted $\varphi:H\to G$ is an injection $\varphi:V(H)\to V(G)$ with the property that $(u,v)\in E(H) \Rightarrow (\varphi(u), \varphi(v))\in E(G)$.  We say $H$ is a \emph{subgraph} of $G$, denoted $H\subseteq G$, if there exists an embedding of $H$ into $G$.  We say that $H$ is a \emph{spanning subgraph} of $G$ if there exists a surjective embedding of $H$ into $G$.  

Given an oriented graph $G$ and $S\subseteq V(G)$, we let $G[S]$ be the subgraph induced by $S$.  Given $v\in V(G)$, we let $G-v=G[V(G)\setminus \{v\}]$, and more generally, given $S\subseteq V(G)$, we let $G-S=G[V(G)\setminus S]$.  If $G$ is an oriented graph and $(u,v)\in E(G)$, then we say that $v$ is an \emph{out-neighbor} of $u$, and that $u$ is an \emph{in-neighbor} of $v$. Given $v\in V(G)$, the \emph{out-neighborhood} of $v$, written $N^+(v)$, is the set of out-neighbors of $v$ in $V(G)$, and the \emph{in-neighborhood} of $v$, written $N^-(v)$ is the set of in-neighbors of $v$ in $V(G)$. Throughout, we use $+$ and $-$ interchangeably with `out' and `in' respectively. For each $\diamond\in\{+,-\}$, the \emph{$\diamond$-degree} of $v$ in $G$ is $d^\diamond(v)=|N^\diamond(v)|$ and the \emph{common $\diamond$-neighborhood} of a set $X\subseteq V(G)$ is $N^\diamond(X)=\bigcap_{v\in X}N^\diamond(v)$.

Given an acyclic oriented graph $G$ and $u\in V(G)$, let $$\Gamma^+(u)=\{v\in V(G): \text{there exists a directed path from $u$ to $v$}\}$$ and let $$\Gamma^-(u)=\{v\in V(G): \text{there exists a directed path from $v$ to $u$}\}$$ (equivalently, $\Gamma^\diamond(u)$ is the $\diamond$-neighborhood of $u$ in the transitive, reflexive closure of $G$).

Given a strict total order $\tau=(V,\prec)$, let $K_{\tau}$ be the tournament on $V$ where $(u,v)\in E(K_\tau)$ if and only if $u\prec v$.  We write $\tau^*$ to be the converse of $\tau$; that is, $\tau^*=(V,\succ)$.  We say that $K_{\tau}$ and $K_{\tau^*}$ are the transitive tournaments of type-$\tau$.

Throughout the paper, we assume the axiom of choice whenever necessary.  We use the von Neumann definition of ordinals where an ordinal is the strictly well-ordered set of all smaller ordinals.  Thus, given an ordinal $\lambda$, the definition of $K_{\lambda}$ and $K_{\lambda^*}$ is given in the previous paragraph.  As is standard, we let $\omega$ be the first infinite ordinal.  Given a cardinal $\kappa$, we view $\kappa$ as the smallest ordinal of cardinality $\kappa$.  We let $\kappa^+$ be the smallest cardinal greater than $\kappa$, and we let $2^\kappa$ be the cardinality of the power set of $\kappa$.  

We say that an infinite cardinal $\kappa$ is \emph{regular} if it cannot be written as the union of fewer than $\kappa$ many sets, each of cardinality less than $\kappa$.  If $\kappa$ is not regular, we say that it is \emph{singular}.  Given a limit ordinal $\alpha$, 
the \emph{cofinality} of $\alpha$, denote $\cof(\alpha)$, is the smallest cardinality of a cofinal subset of $\alpha$ ($S\subseteq \alpha$ is \emph{cofinal} if for all $\beta < \alpha$, there exists a $\gamma \in S$ such that $\beta \leq \gamma$).  With this terminology, we may equivalently say that $\kappa$ is singular if and only if $\cof(\kappa)<\kappa$.

By $\mathbb{N}$ we mean the set of positive integers. 

\section{Countably-infinite oriented graphs}\label{sec:count}

\subsection{Characterizing unavoidable oriented graphs}\label{sec:charactizing-unavoidability}

In this section we will prove that \ref{c2}$\Leftrightarrow$\ref{c3}$\Leftrightarrow$\ref{c4}.

The following lemma can essentially be found in \cite[2.15.1]{F} (we give the more general statement later as Observation \ref{obs:u3u4} and Proposition \ref{prop:MP}).  However, since the proof of this special case is straightforward, we give it here.  

\begin{lemma}\label{lem:order1}
For every countably-infinite oriented graph $G$, \ref{c3}$\Leftrightarrow$\ref{c4}; that is, $G\subseteq K_{\omega}$ and $G\subseteq K_{\omega^*}$ if and only if $G$ is acyclic, locally-finite, and has no infinite directed paths.
\end{lemma}

\begin{proof}
\ref{c3}$\Rightarrow$\ref{c4}: First note that if $G$ has a cycle, then $G\not\subseteq K_{\omega}$ and $G\not\subseteq K_{\omega^*}$.  If $G$ has a vertex of infinite in-degree, then $G\not\subseteq K_{\omega}$, and if $G$ has a vertex of infinite out-degree, then $G\not\subseteq K_{\omega^*}$.  If $G$ has an infinite directed path with the first vertex having out-degree 0, then $G\not\subseteq K_{\omega}$, and if $G$ has an infinite directed path with the first vertex having in-degree 0, then $G\not\subseteq K_{\omega^*}$.

\ref{c4}$\Rightarrow$\ref{c3}: Now suppose $G$ is acyclic, locally-finite, and has no infinite directed paths.  Let $(v_i)_{i\in \omega}$ be an enumeration of $V(G)$.  For all $v_i\in V(G)$, let $f(v_{i})=\max\{j: v_{j}\in \Gamma^-(v_{i})\}$ and note that by the assumptions on degrees and the fact that there are no infinite directed paths we have that $f(v_{i})$ is finite.

We will produce an embedding $\varphi:G\to K_{\omega}$ as follows:  Let $i_0\in \omega$ be minimum such that $v_{i_0}$ has in-degree 0 in $G$ and set $\varphi(v_{i_0})=0$.  On step $j\geq 1$, let $i_j\in \omega$ be minimum such that $v_{i_j}$ has in-degree 0 in $G-\{v_{i_0}, \dots, v_{i_{j-1}}\}$ and set $\varphi(v_{i_j})=j$.  

In the resulting embedding we have the property that for all $i\in \omega$, $\varphi^{-1}(i)$ has no in-neighbors $v$ with $\varphi(v)>i$ and thus we have the desired embedding, provided that $V(G)=\dom \varphi$.  However, this holds because for all $v\in V(G)$, we will assign a value for $\varphi(v)$ by step $f(v_i)$.

An embedding of $G$ into $K_{\omega^*}$ can be constructed similarly.  
\end{proof}

The following lemma shows that it is possible to use a Ramsey-type result to get a result about transitive tournaments.\footnote{On the surface, it is strictly stronger because it is possible to order the vertices of the tournament $K$ and have a copy of $K_{\tau}$ or $K_{\tau^*}$ which doesn't obey the ordering.}  We state it in a form which is more general than what is needed for this section because we will make reference to it again later in a more general setting.  While this folklore result surely appears in the literature, we include a proof for completeness.

\begin{lemma}\label{lem:ramsey}
Let $\sigma=(S,<_\sigma)$ be a strict total order.  If for every 2-coloring of the pairs of elements in $S$, there exists a monochromatic copy of the strict total order $\tau=(T, <_\tau)$, then for every tournament $K$ of cardinality $|S|$ we have $K_{\tau}\subseteq K$ or $K_{\tau^*}\subseteq K$.
\end{lemma}

\begin{proof}
Let $K$ be a tournament of cardinality $|S|$ and take an arbitrary bijection $\varphi:V(K)\to S$.  For $(u,v)\in E(K)$, if $\varphi(u)<_\sigma \varphi(v)$, color $(u,v)$ red, and if $\varphi(u)>_\sigma \varphi(v)$, color $(u,v)$ blue.  By the assumption, there is a monochromatic copy of $\tau$.  If the copy is red, then we have $K_{\tau}\subseteq K$, and if the copy is blue, then we have $K_{\tau^*}\subseteq K$.
\end{proof}

\begin{lemma}\label{thm:unavoidable-characterization}
For every countably-infinite oriented graph $G$, \ref{c2}$\Leftrightarrow$\ref{c3}; that is, $G$ is unavoidable if and only if $G\subseteq K_{\omega}$ and $G\subseteq K_{\omega^*}$.
\end{lemma}

\begin{proof}
\ref{c2}$\Rightarrow$\ref{c3}: If $G\not\subseteq K_{\omega}$ or $G\not\subseteq K_{\omega^*}$, then $G$ is avoidable.  

\ref{c3}$\Rightarrow$\ref{c2}:  Suppose $G\subseteq K_{\omega}$ and $G\subseteq K_{\omega^*}$ and let $K$ be a countably-infinite tournament.  By Ramsey's theorem and Lemma \ref{lem:ramsey} (with $\sigma=\tau=\omega$) we have $K_{\omega}\subseteq K$ or $K_{\omega^*}\subseteq K$.  Either way, we have $G\subseteq K$.  
\end{proof}

\subsection{Unavoidable oriented graphs are strongly unavoidable}\label{sec:strong-unavoidability}

In this section we prove \ref{c4}$\Rightarrow$\ref{c1}.  To prepare for the proof, we first need a structural result about oriented graphs satisfying \ref{c4}.  

Given an countably-infinite acyclic weakly-connected oriented graph $G$, a \emph{$\pm$-partition} of $G$ is a partition $\{C_i:i\in \NN\}$ of $V(G)$ such that the following properties hold:
\stepcounter{capitalcounter}
\begin{enumerate}[label = {\bfseries{\Alph{capitalcounter}\arabic{enumi}}}]
\item\label{closure-1} For all $i\in \NN$, $C_i$ is finite and non-empty. 
\item\label{closure-2} For all $(u,v)\in E(G)$, there exists $i\in \NN$ such that $\{u,v\}\subseteq C_{i}\cup C_{i+1}$.
\item\label{closure-3} If $i$ is odd, then every vertex in $C_i$ has in-degree 0 to $C_{i-1}\cup C_{i+1}$, and if $i$ is even, then every vertex in $C_i$ has out-degree 0 to $C_{i-1}\cup C_{i+1}$.  
\item\label{closure-4} If $i$ is odd, then there exists a vertex in $C_i$ with in-degree 0 in $G$, and if $i$ is even there exists a vertex in $C_i$ with out-degree 0 in $G$.
\end{enumerate}
If $i$ is odd, we say that $C_i$ has \emph{type} $+$, and if $i$ is even, we say that $C_i$ has \emph{type} $-$.  

Likewise one can define a \emph{$\mp$-partition} by switching every instance of in/out in the above definition.  We note that a similar definition for finite oriented trees was given by Dross and Havet \cite{DH}.

\begin{lemma}\label{lem:cover2}Let $G$ be a countably-infinite oriented graph.
If $G$ is weakly-connected, acyclic, locally-finite, and has no infinite directed paths, then for every vertex $v$ of in-degree 0, $G$ has a $\pm$-partition with $C_1=\{v\}$, and for every vertex $v$ of out-degree 0, $G$ has a $\mp$-partition with $C_1=\{v\}$. 
\end{lemma}

\begin{proof}
Since $G$ is acyclic and has no infinite directed paths, the set of vertices with in-degree 0 is non-empty; let $v$ be a vertex of in-degree 0 and set $C_1=\{v\}$.  For even $i>1$, let $C_i=\left(\bigcup_{v\in C_{i-1}}\Gamma^+(v)\right)\setminus (C_{i-1}\cup C_{i-2})$, and for odd $i>1$, let $C_i=\left(\bigcup_{v\in C_{i-1}}\Gamma^-(v)\right)\setminus (C_{i-1}\cup C_{i-2})$. Note that since $G$ is locally-finite, and has no infinite directed paths, each $C_i$ is finite. In addition, because $G$ is weakly-connected, $\{C_i:i\in\NN\}$ is a partition of $V(G)$, and each $C_i$ is non-empty. Therefore, \ref{closure-1} holds.

Suppose, for some $i<j$, that $(u,v)\in E(G)$ with $u\in C_i$ and $v\in C_j$. Then, we must have that $i$ is odd (else $v\in C_i$) and $j=i+1$. On the other hand, if $(v,u)\in E(G)$ is such that $u\in C_i$ and $v\in C_j$ for some $i<j$, then we must have that $i$ is even (else $v\in C_i$) and $j=i+1$. Therefore, we deduce that \ref{closure-2} and \ref{closure-3} hold.

Finally, note that since $C_i$ is finite and $G$ is acyclic, $G[C_i]$ has a vertex $u_i$ of in-degree 0 in $G[C_i]$ and a vertex $v_i$ of out-degree 0 in $G[C_i]$.  Thus, by \ref{closure-3}, if $i$ is even, then $v_i$ has out-degree 0 in $G$, and if $i$ is odd, then $u_i$ has in-degree 0 in $G$. Therefore, \ref{closure-4} holds, and $\{C_i:i\in\NN\}$ is a $\pm$-partition with $C_1=\{v\}$.

Likewise by switching every instance of in/out in the above proof, we get that for every vertex $v$ of out-degree 0, $G$ has a $\mp$-partition with $C_1=\{v\}$.
\end{proof}

\begin{theorem}\label{thm:unavoidable-bijection}
For every countably-infinite oriented graph $G$, \ref{c4}$\Rightarrow$\ref{c1}; that is, if $G$ is acyclic, locally-finite, and has no infinite directed paths, then $G$ is a spanning subgraph of every countably-infinite tournament.  
\end{theorem}

\begin{proof}
Suppose $G$ is acyclic, locally-finite, and has no infinite directed paths. If $G$ is not weakly-connected, we can make it so while maintaining the three properties (say by choosing a vertex $v_i$ from each component $H_i$ of $G$ and adding an anti-directed path on $v_1,v_2,\dots$). Let $K$ be a countably-infinite tournament and let $(u_i)_{i\in \NN}$ be an enumeration of $V(K)$. Define $\ast_1,\ast_2,\ast_3,\ldots$ inductively by
\begin{equation*}
\ast_i=
    \begin{cases*}
        + & if $\left(\bigcap_{j=1}^{i-1}N^{\ast_j}(u_j)\right)\cap N^+(u_i)$ is infinite, \\
        - & otherwise.
    \end{cases*}
\end{equation*}
Let $V^+=\{u_i\in V(K): \ast_i=+\}$ and let $V^-=\{u_i\in V(K): \ast_i=-\}$.  The key property is that for all $\diamond, \ast\in \{+,-\}$ and all finite non-empty subsets $X\subseteq V^\diamond$ and $Y\subseteq V^{\ast}$, $N^\diamond(X)\cap N^{\ast}(Y)$ is infinite. (A more standard approach to assigning the $\ast_i$ would have been to choose an ultrafilter on $\NN$ and let $\ast_i=\diamond$ iff $N^\diamond(u_i)$ is in the ultrafilter.  We note that our assignment of $\ast_i$ without the use of ultrafilters is inspired by the proof of \cite[Lemma~3.4]{L}.)

If $\ast_1=+$, then we choose a vertex $v_1\in V(G)$ with in-degree 0 and apply Lemma~\ref{lem:cover2} to get a $\pm$-partition $\{C_i:i\in \NN\}$ of $G$ with $C_1=\{v_1\}$.  If $\ast_1=-$, then we choose a vertex $v_1\in V(G)$ with out-degree 0 and apply Lemma~\ref{lem:cover2} to get a $\mp$-partition $\{C_i:i\in \NN\}$ of $G$ with $C_1=\{v_1\}$.  We may suppose without loss of generality that $\ast_1=+$ and thus we choose a vertex $v_1\in V(G)$ with in-degree 0 and apply Lemma~\ref{lem:cover2} to get a $\pm$-partition $\{C_i:i\in \NN\}$ of $G$ with $C_1=\{v_1\}$.
Finally, define
\begin{equation*}
\diamond_i=
    \begin{cases*}
        + & if i \text{ is odd}, \\
        - & if i \text{ is even}
    \end{cases*}
\end{equation*}
and note that $\diamond_i$ simply describes the type of the set $C_i$.  

We construct a sequence $i_1\leq i_2\leq \ldots$, growing an embedding $\varphi:G[\cup_{i\in[i_j]}C_i]\to K$ as we do so, such that following properties hold for every $j\in\NN$.
\stepcounter{capitalcounter}
\begin{enumerate}[label = {\bfseries{\Alph{capitalcounter}\arabic{enumi}}}]
\item\label{embed-1} $\{u_1,\ldots,u_j\}\subseteq \varphi(\cup_{i\in[i_j]}C_i)$, and
\item\label{embed-2} $\varphi(C_{i_j})\subseteq V^{\diamond_{i_j}}$.
\end{enumerate}
If such a sequence exists, then by \ref{embed-1}, the resulting embedding $\varphi:G\to K$ proves the theorem.

We initially set $i_1=1$ and $\varphi(v_1)=u_1$. Then, given $i_{j-1}$ and $\varphi:G[\cup_{i\in[i_{j-1}]}C_i]\to K$ satisfying \ref{embed-1} and \ref{embed-2}, we proceed as follows.

If $u_j\in \varphi(\cup_{i\in[i_{j-1}]}C_i)$, then set $i_j=i_{j-1}$ (trivially, \ref{embed-1} and \ref{embed-2} are satisfied). Otherwise, 
by \ref{embed-2} we have that $U_{j-1}:=N^{\ast_j}(u_j)\cap N^{\diamond_{i_{j-1}}}(\varphi(C_{i_{j-1}}))$ is infinite.  If $U_{j-1}\cap V^+$ is infinite, set $i_j$ to be the smallest integer at least $i_{j-1}+5$ with $\diamond_{i_j}=+$ (i.e. the smallest odd integer at least $i_{j-1}+5$). Otherwise, $U_{j-1}\cap V^-$ is infinite and we set $i_j$ to be the smallest integer at least $i_{j-1}+5$ with $\diamond_{i_j}=-$ (i.e. the smallest even integer at least $i_{j-1}+5$).  We now embed the acyclic finite subgraph $G[C_{i_{j-1}+1}\cup\ldots\cup C_{i_j}]$ into the infinite tournament $K[\{u_j\}\cup (U_{j-1}\cap V^{\diamond_{i_j}})]$ in such a way that if $\ast_j=\diamond_{i_{j-1}}$, then we will choose a vertex $v_j\in C_{i_{j-1}+2}$ which only has $\ast_j$-neighbors and embed $v_j$ to $u_j$, and if $\ast_j\neq \diamond_{i_{j-1}}$, then we will choose a vertex $v_j\in C_{i_{j-1}+3}$ which only has $\ast_j$-neighbors and embed $v_j$ to $u_j$.  Thus \ref{embed-1} is satisfied.  Also note that by construction, every vertex in $C_{i_j}$ is embedded into $V^{\diamond_{i_j}}$, so \ref{embed-2} is satisfied.\footnote{In the above paragraph, it is instructive to have a specific example, so suppose $\varphi(C_{i_{j-1}})\subseteq V^+$ (i.e. $i_{j-1}$ is odd), $u_j\in V^-$, and $(N^-(u_j)\cap N^+(\varphi(C_{i_{j-1}}))\cap V^-$ is infinite.  In this case we would set $i_j=i_{j-1}+5$ (note that $i_j$ is even), embed a vertex from $C_{i_{j-1}+3}$ ($i_{j-1}+3$ is also even) with in-degree 0 to $u_j$ and embed the rest of $C_{i_{j-1}+1}\cup \dots \cup C_{i_{j}}$ into $(N^-(u_j)\cap N^+(\varphi(C_{i_{j-1}}))\cap V^-$.  Note that since $i_j$ is even and $\varphi(C_{i_j})\subseteq V^-$, \ref{embed-2} is satisfied.}
\end{proof}

\section{Uncountable oriented graphs}\label{sec:uncount}

Let $\kappa$ be a cardinal and let $G$ be an oriented graph with $|V(G)|\leq \kappa$.  We say that $G$ is \emph{$\kappa$-unavoidable} if $G$ is contained in every tournament $K$ with $|V(K)|= \kappa$, otherwise we say that $G$ is \emph{$\kappa$-avoidable}.  We say that $G$ is \emph{strongly $\kappa$-unavoidable} if $G$ is a spanning subgraph of every tournament $K$ with $|V(K)|= \kappa$.  

The purpose of this section is to discuss the relationships between the following analogues of \ref{c1}-\ref{c4} for oriented graphs $G$ with $|V(G)|=\kappa$ where $\kappa$ is uncountable.  

\begin{enumerate}[label = {(\bfseries{U\arabic{enumi}})}]
\item\label{u1} $G$ is strongly $\kappa$-unavoidable.
\item\label{u2} $G$ is $\kappa$-unavoidable.
\item\label{u3} $G\subseteq K_{\kappa}$ and $G\subseteq K_{\kappa^*}$
\item\label{u4} $G$ is acyclic, has no infinite directed paths, and every vertex has degree less than $\kappa$ in the transitive closure of $G$.
\end{enumerate}

As before, we trivially have \ref{u1}$\Rightarrow$\ref{u2}$\Rightarrow$\ref{u3}, and as before it is not hard to see that \ref{u3}$\Rightarrow$\ref{u4}.  At this point the reader may wonder why we have \ref{u4} instead of the following:
\begin{enumerate}
\item[{\crtcrossreflabel{(\bfseries{U4$'$})}[u4']}] $G$ is acyclic, has no infinite directed paths, and every vertex  has degree less than $\kappa$.
\end{enumerate}

The reason is that if $\kappa$ is a regular cardinal, then \ref{u4} and \ref{u4'} are equivalent, but if $\kappa$ is singular, they are not (see Observation \ref{obs:singular1}).  

We prove the following results.

\begin{theorem}\label{thm:u3u4}
For all infinite cardinals $\kappa$, \ref{u3}$\Leftrightarrow$\ref{u4}.
\end{theorem}

\begin{theorem}\label{thm:reg}
If $\kappa$ is an uncountable regular cardinal, then \ref{u1}$\Leftrightarrow$\ref{u2}.
\end{theorem}

\begin{theorem}\label{thm:aleph1}
If $\kappa=\aleph_1$, then \ref{u4}$\Rightarrow$\ref{u1}.
\end{theorem}

\begin{theorem}\label{thm:strong}
If $\kappa$ is a strongly inaccessible cardinal, then \ref{u4}$\Rightarrow$\ref{u1}.
\end{theorem}

The natural generalization of Theorem \ref{thm:unavoidable-bijection} to the uncountable case would be that \ref{u4} implies \ref{u1} for all $\kappa$.  However, the following theorem shows that if $\lambda$ is an infinite cardinal at which GCH fails (i.e., $2^{\lambda} > \lambda^{+}$)
then there is a graph $G$ with $|V(G)| = 2^{\lambda}$ (trivially extending the graph $H$ from the theorem) satisfying \ref{u4} but not \ref{u2}.

\begin{theorem}\label{thm:GCH}
	For all infinite cardinals $\lambda$, there exists a transitive tournament $T$ on $2^{\lambda}$ and an acyclic oriented graph $H$ with no infinite directed paths and $|V(H)|=\lambda^+$ such that $H$ does not embed into $T$. 
\end{theorem}

In light of Theorems \ref{thm:reg} and \ref{thm:GCH}, one pressing open question is whether the assumption that a graph $G$ embeds into every transitive tournament of cardinality $\kappa$ gives any further structural information about $G$ beyond what \ref{u3} provides (in order to find a potential replacement for \ref{u4}).

Theorem \ref{thm:GCH} also leaves open the question of whether (or when) the GCH implies the equivalence of \ref{u2} and \ref{u4}. One could also ask the following question (which necessarily falls short of giving a characterization for \ref{u1}).

\begin{question}\label{q:lambda}
Let $\kappa$ be an uncountable cardinal.  For which cardinals $\mu\leq \kappa$ and $\lambda<\kappa$ is the following true?
\begin{enumerate}
\item\label{q1} If $G$ is an oriented graph with $|V(G)|=\mu$ such that $G$ is acyclic, has no infinite directed paths, and every vertex has degree at most $\lambda$, then $G$ is $\kappa$-strongly-unavoidable.

\item\label{q2} If $\kappa$ is singular and $G$ is an oriented graph with $|V(G)|=\kappa$ such that $G$ is acyclic, has no infinite directed paths, and there exists a $\nu<\kappa$ such that every vertex has degree at most $\nu$, then $G$ is $\kappa$-strongly-unavoidable.\footnote{This is weaker than \ref{u4}$\Rightarrow$\ref{u1} for singular $\kappa$ because there is a uniform upper bound on the degrees.}
\end{enumerate}
\end{question}

\subsection{Characterizing oriented graphs which embed into both $K_{\kappa}$ and $K_{\kappa^*}$}

In this section we prove Theorem \ref{thm:u3u4}.  We begin by observing that \ref{u3} implies \ref{u4}. 

\begin{observation}\label{obs:u3u4}
Let $\kappa$ be an infinite cardinal and let $G$ be an oriented graph with $|V(G)|\leq \kappa$.  If $G\subseteq K_{\kappa}$ and $G\subseteq K_{\kappa^*}$, then $G$ is acyclic, $G$ has no infinite directed paths, and every vertex has degree less than $\kappa$ in the transitive closure of $G$.  
\end{observation}

\begin{proof}
Note that $K_{\kappa}$ and $K_{\kappa^*}$ are acyclic, $K_{\kappa}$ has no infinite backwards directed paths and every vertex in $K_{\kappa}$ has in-degree less than $\kappa$, and $K_{\kappa^*}$ has no infinite forwards directed paths and every vertex in $K_{\kappa^*}$ has out-degree less than $\kappa$.  So if $G\subseteq K_{\kappa}$ and $G\subseteq K_{\kappa^*}$, then the transitive closure $\vec{G}$ of $G$ also satisfies $\vec{G}\subseteq K_{\kappa}$ and $\vec{G}\subseteq K_{\kappa^*}$ and thus $G$ is acyclic, has no infinite forward or backward directed paths, and every vertex of $\vec{G}$ has in-degree and out-degree less than $\kappa$.
\end{proof}

The reverse implication can essentially be found in \cite[2.15.1]{F}, where it is attributed to Milner and Pouzet.  

\begin{proposition}\label{prop:MP}
Let $\kappa$ be an infinite cardinal and let $G$ be an acyclic transitive oriented graph with $|V(G)|\leq \kappa$.  If $G$ has no infinite directed paths and every vertex has degree less than $\kappa$, then $G\subseteq K_{\kappa}$ and $G\subseteq K_{\kappa^*}$.
\end{proposition}

By applying Proposition \ref{prop:MP} to the transitive closure of $G$, we obtain \ref{u4}$\Rightarrow$\ref{u3}.

\subsection{Regular cardinals}

In this section we prove Theorems \ref{thm:reg}, \ref{thm:aleph1}, and \ref{thm:strong}.  We begin with a few preliminary results, the first of which just follows from the definition of regular cardinal.

\begin{observation}\label{obs:reg_components}
Let $\kappa$ be an uncountable cardinal and let $G$ be an acyclic oriented graph with $|V(G)|= \kappa$.  
If $\kappa$ is a regular cardinal and every vertex in $G$ has degree less than $\kappa$, then the transitive closure of $G$ has $\kappa$ many connected components each of which has cardinality less than $\kappa$.
\end{observation}

Next is the key lemma we will use for obtaining a surjective embedding.  

\begin{lemma}\label{lem:key1}
Let $\kappa$ be an infinite cardinal, let $G$ be an oriented graph with $|V(G)|\leq \kappa$, and let $K$ be a tournament with $|V(K)|=\kappa$.  If $G$ is $\kappa$-unavoidable, then for each $v\in V(K)$ there exists an embedding $\varphi:G\to K$ such that $v\in \varphi(V(G))$.
\end{lemma}

\begin{proof}
Let $v\in V(K)$.  Either $v$ is incident with $\kappa$ many out-edges or $\kappa$ many in-edges.  If it is the former, let $u\in V(G)$ have in-degree 0 and set $K'=K[N^+(v)]$. If it is the latter, let $u\in V(G)$ have out-degree 0 and set $K'=K[N^-(v)]$.  In either case set $\varphi(u)=v$.  Since $G$ is $\kappa$-unavoidable, we can embed $G-u$ in $K'$, which gives us an embedding $\varphi:G\to K$ such that $v\in \varphi(V(G))$.
\end{proof}

The following lemma gives a sufficient condition for obtaining a surjective embedding.  

\begin{lemma}\label{lem:key2}
Let $\kappa$ be an uncountable regular cardinal and let $G$ be an oriented graph with $|V(G)|=\kappa$.
If $G$ has $\kappa$ many connected components of cardinality less than $\kappa$, each of which is $\kappa$-unavoidable, then $G$ is strongly $\kappa$-unavoidable.
\end{lemma}

\begin{proof}
Let $K$ be a tournament of cardinality $\kappa$, let $(v_\gamma)_{\gamma\in \kappa}$ be an enumeration of $V(K)$, and let $(G_\alpha)_{\alpha\in \kappa}$ be an enumeration of the components of $G$.  We construct a surjective embedding $\varphi \colon G\to K$ as the union of a recursively chosen collection of embeddings $\varphi_{\alpha} \colon G_{\alpha} \to K$ with disjoint ranges.

Suppose that $\alpha \in \kappa$ and that $\varphi_{\beta}$ ($\beta < \alpha$) have been chosen. Since $\kappa$ is regular and each $G_{\beta}$ has cardinality less than $\kappa$, $K_{\alpha} = K - \bigcup_{\beta < \alpha}\ran(\varphi_{\beta})$ has cardinality $\kappa$. 
Let $\gamma_{\alpha}\in \kappa$ be minimal such that $v_{\gamma_{\alpha}}$ is in $K_{\alpha}$.  
Apply Lemma \ref{lem:key1} to find an embedding $\varphi_{\alpha} \colon G_{\alpha} \to K_\alpha$ in such a way that $v_{\gamma_\alpha}$ is in the range of $\varphi_{\alpha}$. 
Having chosen $\varphi_{\alpha}$ for all $\alpha < \kappa$, let $\varphi = \bigcup_{\alpha < \kappa}\varphi_{\alpha}$.  Since every vertex $v_{\gamma}$ is in the range of some $\varphi_{\alpha}$, this completes the proof.
\end{proof}

\begin{proof}[Proof of Theorem \ref{thm:reg}]
Since $G$ is $\kappa$-unavoidable, by Observation \ref{obs:u3u4}, $G$ is acyclic, has no infinite directed paths, and every vertex (in the transitive closure of $G$) has degree less than $\kappa$.  So by Observation \ref{obs:reg_components}, $G$ has $\kappa$-many connected components all of which are $\kappa$-unavoidable.  Thus by Lemma \ref{lem:key2}, $G$ is strongly $\kappa$-unavoidable.  
\end{proof}

We note that the previous proof does not apply to singular cardinals because for singular cardinals $\kappa$, \ref{u4} does not imply that $G$ has $\kappa$ many components (in fact $G$ can even be connected in this case -- see Observation \ref{obs:singular2}) and thus we cannot make use of Lemma \ref{lem:key2}.

In Section \ref{sec:charactizing-unavoidability} we showed that \ref{c3}$\Rightarrow$\ref{c2} by using Ramsey's theorem.  If there was an analogue of Ramsey's theorem for the uncountable case which implied \ref{u3}$\Rightarrow$\ref{u2}, then from Theorem \ref{thm:u3u4} and Theorem \ref{thm:reg} we would already have that \ref{u1}-\ref{u4} are equivalent.  However, it is known by a result of Laver (see \cite{GS}) that there is a tournament of cardinality $\aleph_1$ which doesn't contain any transitive subtournament of cardinality $\aleph_1$.  (Furthermore, Justin Moore pointed out to us that results of Rinot and Todorcevic \cite{RinTod} can be used to show that the corresponding fact holds at the successor of any regular cardinal.)

On the other hand a result of Baumgartner and Hajnal \cite{BH} shows that the next best thing is true; that is, for every countable ordinal $\alpha$ and every 2-coloring of the pairs in $\omega_1$, there is a monochromatic copy of $\alpha$, which by Lemma \ref{lem:ramsey} implies that every tournament of cardinality $\aleph_1$ contains $K_{\alpha}$ or $K_{\alpha^*}$. (Note that while Baumgartner and Hajnal use some deep set-theoretic ideas to establish their result in ZFC, Galvin \cite{G} later gave a direct combinatorial proof in ZFC.)

\begin{theorem}[Baumgartner-Hajnal; Galvin]\label{thm:BHG}
	Let $K$ be an uncountable tournament.  For every countable ordinal $\alpha$ we have $K_{\alpha}\subseteq K$ or $K_{\alpha^*}\subseteq K$.
\end{theorem}

Using Theorem \ref{thm:BHG} we are able to establish \ref{u4}$\Rightarrow$\ref{u1} for $\kappa=\aleph_1$.  First we need another variant on Szpilrajn's extension theorem (see the comment before Lemma \ref{lem:order1}) which is essentially equivalent to the fact that every well-founded partial order of cardinality $\kappa$ can be extended to a well-order of cardinality $\kappa$.  A sketch of a proof can be found in \cite[2.9.2]{F}, but we give a full proof for completeness.  

\begin{proposition}\label{prop:1}
	Let $\kappa$ be an infinite cardinal and let $G$ be an acyclic oriented graph with $|V(G)|\leq \kappa$.  If $G$ has no infinite backward directed paths, then there exists an ordinal $\beta$ of cardinality at most $\kappa$ such that $G\subseteq K_{\beta}$.  Likewise if $G$ has no infinite forward directed paths, then there exists an ordinal $\beta$ of cardinality at most $\kappa$ such that $G\subseteq K_{\beta^*}$. 
\end{proposition}

\begin{proof}
	We prove the first conclusion; the proof of the second is essentially the same. Let $h$ be the function from $G$ to the ordinals defined recursively by the formula \[h(x)=\sup\{h(w)+1: w\in \Gamma^-(x)\},\] for all $x\in V(G)$.  Given an ordinal $\alpha$, let $V_\alpha$ be $\{x\in V(G): h(x)=\alpha\}$.  Let $\beta_0=\{\alpha: V_\alpha\neq \emptyset\}$ (which is an ordinal).  Note that $\beta_{0}$ has cardinality at most $\kappa$, since otherwise $G$ has more than $\kappa$-many non-empty levels, which means $|V(G)|> \kappa$, giving a contradiction. Let $\beta = \sum_{\alpha< \beta_{0}}|V_{\alpha}|$. Then again $|\beta| \leq \kappa$. 
	
	For each $\alpha< \beta_{0}$, let $(v^\alpha_\gamma)_{\gamma<|V_{\alpha}|}$ be an enumeration of $V_{\alpha}$.  Then there exists an injection $\varphi \colon G\to K_{\beta}$ such that $\varphi(v^{\alpha'}_{\gamma'})<\varphi(v^\alpha_\gamma)$ if and only if $\alpha'<\alpha$, or $\alpha'=\alpha$ and $\gamma'<\gamma$.
\end{proof}

We now prove that \ref{u4}$\Rightarrow$\ref{u1} for $\kappa=\aleph_1$.  

\begin{proof}[Proof of Theorem \ref{thm:aleph1}]
By Observation \ref{obs:reg_components}, $G$ has $\aleph_1$ many components each of cardinality at most $\aleph_0$.  By Proposition \ref{prop:1}, for every component $H$ of $G$ there exists a countable ordinal $\alpha$ such that $H$ embeds into both $K_\alpha$ and $K_{\alpha^*}$.  By Theorem \ref{thm:BHG}, every tournament of cardinality $\aleph_1$ contains either $K_\alpha$ or $K_{\alpha^*}$, and thus $H$ is $\aleph_1$-unavoidable.

So $G$ has $\aleph_1$ many components each of which is $\aleph_1$-unavoidable and thus by applying Lemma \ref{lem:key2}, we have that $G$ is strongly $\aleph_1$-unavoidable.  
\end{proof}

\begin{remark}
The proof of Theorem \ref{thm:aleph1} also shows that Question \ref{q:lambda}.(i) has a positive answer whenever $\lambda=\aleph_0$.
\end{remark}

Note that the proof of Theorem \ref{thm:aleph1} would generalize to any uncountable regular cardinal $\kappa$ provided there was a generalization of Theorem \ref{thm:BHG} for $\kappa\geq \omega_2$ (that is, for every ordinal $\alpha<\kappa$, every tournament $K$ of cardinality $\kappa$ contains $K_{\alpha}$ or $K_{\alpha^*}$).  As it turns out, this is a problem first raised by Erd\H{o}s and Hajnal \cite[Problem 10]{EH1} and it is still an open question whether it is consistent that any successor cardinal $\kappa \geq \omega_{2}$ can have this property (see \cite[9.1]{FH} and \cite[Question 5.1]{Fo}).  As we shall see in Section \ref{sec:GCH}, consistently this property fails at every double successor cardinal. 

Finally, we consider the case of strongly inaccessible cardinals.  A cardinal $\kappa$ is \emph{strongly inaccessible} if it is regular and a strong limit; that is, $2^{\lambda} < \kappa$ for all $\lambda < \kappa$.  Taking the place of Theorem \ref{thm:BHG} in this case is the following classical result of Erd\H{o}s \cite{E} (stated here in slightly less general form and using Lemma \ref{lem:ramsey} as we did for Theorem \ref{thm:BHG}).

\begin{theorem}[Erd\H{o}s]\label{thm:E}
For all infinite cardinals $\lambda$, every tournament $K$ of cardinality $(2^\lambda)^+$ contains $K_{\lambda^+}$ or $K_{(\lambda^+)^*}$. 
\end{theorem}

Now it is a routine matter to complete the proof.  

\begin{proof}[Proof of Theorem \ref{thm:strong}]
Since $\kappa$ is regular, Observation \ref{obs:reg_components} implies that $G$ has $\kappa$ many components each of cardinality less than $\kappa$.  By Proposition \ref{prop:1}, every component $H$ of $G$ of cardinality $\lambda$ embeds into $K_{\lambda^+}$ and $K_{(\lambda^+)^*}$.  Since $(2^{\lambda})^+ < \kappa$, Theorem \ref{thm:E} implies that every tournament of cardinality $\kappa$ contains contains $K_{\lambda^+}$ or $K_{(\lambda^+)^*}$, and thus $H$ is $\kappa$-unavoidable. 

So $G$ has $\kappa$ many components each of which is $\kappa$-unavoidable and thus by applying Lemma \ref{lem:key2}, we have that $G$ is strongly $\kappa$-unavoidable.  
\end{proof}

\begin{remark} The proof of Theorem \ref{thm:strong} also shows that Question \ref{q:lambda}.(i) has a positive answer whenever $\kappa \geq (2^{\lambda})^{+}$. 
\end{remark}

\subsection{Singular cardinals}

In this section we collect two observations about the case of singular $\kappa$. 
We first show that \ref{u4} and \ref{u4'} are not equivalent for singular cardinals.  

\begin{observation}\label{obs:singular1}
Let $\kappa$ be a singular cardinal.  There exists a graph $G$ with $|V(G)|=\kappa$ in which $G$ is acyclic, has no infinite directed paths, and every vertex has degree less than $\kappa$, but some vertex has degree $\kappa$ in the transitive closure of $G$.  
\end{observation}

\begin{proof}
Let $\langle \lambda_\alpha: \alpha<\cof(\kappa)\rangle$ be an increasing cofinal sequence of regular cardinals less than $\kappa$.
	
Let $X=\{x_\alpha:\alpha<\cof(\kappa)\}$ be a set of vertices such that each $x_\alpha$ has a distinct set of $\lambda_\alpha$ many in-neighbors and let $Z=\{z_\alpha:\alpha<\cof(\kappa)\}$ be a set of vertices such that each $z_\alpha$ has a distinct set of $\lambda_\alpha$ many out-neighbors.  Then add a vertex $y$ such that $N^-(y)=X$ and $N^+(y)=Z$.  Call the resulting oriented graph $G$.  Note that $y$ has in-degree and out-degree equal to $\cof(\kappa)<\kappa$ and every other vertex has in-degree and out-degree at most $\lambda_\alpha<\kappa$ for some $\alpha<\cof(\kappa)$.  If say $G\subseteq K_{\kappa}$, then since $K_{\kappa}$ is transitive, the transitive closure, $\vec{G}$, of $G$ satisfies $\vec{G}\subseteq K_{\kappa}$ (likewise if $G\subseteq K_{\kappa^*}$).  However, in $\vec{G}$ it is the case that $y$ has in-degree and out-degree equal to $\kappa$.  So by Observation \ref{obs:u3u4}, we have that $\vec{G}\not\subseteq K_{\kappa}$ and $\vec{G}\not\subseteq K_{\kappa^*}$, a contradiction.
\end{proof}

We next explain why \ref{u4} does not imply that $G$ has $\kappa$-many components in the singular case and at the same time show that having $\kappa$-many components is not a necessary condition for \ref{u2}. 

\begin{observation}\label{obs:singular2}
There exists a graph $G$ with $|V(G)|=\aleph_{\omega}$ in which $G$ is acyclic, has no infinite directed paths, every vertex has degree less than $\aleph_{\omega}$ in the transitive closure of $G$, but $G$ is connected and $\aleph_{\omega}$-unavoidable.
\end{observation}

\begin{proof}
Let $A_n$ $(n \in \omega)$, be disjoint sets, with each $A_n$ of cardinality $\aleph_n$, and let $b_n$ $(n \in \omega)$, be an additional set of vertices. Let $G$ be the graph with vertex set 
	\[\{b_n : n \in \omega\} \cup \bigcup_{n\in\omega}A_{n}\] and the following edges: each $b_n$ points to all the vertices in $A_n$, and $b_n$ points to $b_m$ whenever $|n-m| =1$ and $n$ is odd. Any tournament on $\aleph_\omega$ will have distinct vertices $c_n$ $(n \in \omega)$, where $c_n$ points to at least $\aleph_n$ many elements. Thinning via Ramsey's theorem for pairs of integers, we can assume that either $m < n$ implies that $c_n$ points to $c_m$, or $m < n$ implies that $c_m$ points to $c_n$. Either way we can embed the $b_n$'s into the $c_n$'s in such a way that each $b_n$ is sent to a $c_m$ with $m \geq n$.
\end{proof}

\subsection{\ref{u4} does not imply \ref{u2} in general}\label{sec:GCH}

In this section we proof Theorem \ref{thm:GCH}. 
First we introduce the following graphs satisfying \ref{u4} which will be used in the proof.  Let $\lambda$ be an ordinal and let $H(\lambda)$ be the oriented graph where the vertices of $H(\lambda)$ are pairs $(\alpha, \beta)\in \lambda \times \lambda$, and $H(\lambda)$ contains an edge from $(\alpha, \beta)$ to $(\gamma, \delta)$ if and only if 
$\alpha > \gamma$ and $\beta < \delta$.  Note that $H$ is a transitive acyclic oriented graph with $|V(H)|=|\lambda|$, no infinite directed paths, and every vertex has degree at most $|\lambda|$.

Given an ordinal $\lambda$, the lexicographical order on $2^{\lambda}$ (viewed as the set of functions from $\lambda\to \{0,1\}$) is defined as follows: let $f <_{\lex} g$ if $f\neq g$ and $f(\alpha) < g(\alpha)$ where $\alpha$ is minimal such that $f(\alpha) \neq g(\alpha)$.  This is a linear ordering of $2^{\lambda}$ which contains no ascending or descending sequences of length $\lambda^{+}$. (One way to see this is to prove inductively on $\alpha \leq \lambda$ that all but at most $|\lambda|$ many members of each such sequence would have to have the same restriction to $\alpha$.)

We now show that, for any infinite cardinal $\lambda$, $H(\lambda^{+})$ does not embed into the tournament on $2^{\lambda}$ induced by the lexicographical order. This shows that if the GCH fails at $\lambda$ (i.e., if $2^{\lambda} \geq \lambda^{++}$), then \ref{u4} does not imply \ref{u2} for graphs of cardinality $\lambda^{++}$ (or, for that matter, for graphs of any cardinality in the interval $[\lambda^{++}, 2^{\lambda}]$). 

\begin{proof}[Proof of Theorem \ref{thm:GCH}]
Let $H:=H(\lambda^+)$ and note that $H$ is a transitive acyclic oriented graph with $|V(H)|=\lambda^+$, $H$ has no infinite directed paths, and every vertex has degree at most $\lambda^+$.  Let $K:=K_{\mathrm{lex}(2^\lambda)}$ be the tournament on $2^\lambda$ where $(f,g)$ is an edge of $K$ if and only if $f <_{\rm{lex}} g$.  We now show that there is no embedding of $H$ into $K$.

Suppose toward a contradiction that $\varphi \colon \lambda^{+} \times \lambda^{+} \to 2^{\lambda}$ were such an embedding; that is, for all $\alpha, \beta, \gamma, \delta \in \lambda^+$,  $$((\alpha, \beta), (\gamma, \delta)) \in E(H) \Rightarrow \varphi(\alpha, \beta) <_{\rm{lex}} \varphi(\gamma, \delta).$$  
By the way $H$ is defined, this implies that  
\begin{equation}\label{eq:H}
\alpha>\gamma \text{ and } \beta<\delta\Rightarrow \varphi(\alpha, \beta) <_{\rm{lex}} \varphi(\gamma,\delta).
\end{equation}

For each $\beta \in \lambda^{+}$, let 
	\begin{itemize}
		\item $X_{\beta}=\{ \varphi(\alpha, \beta) : \alpha \in \lambda^{+}\}$ and 
		\item $S_{\beta}$ be the set of $x \in 2^{\lambda}$ which are $<_{\rm{lex}}$-below $\lambda^{+}$ many members of $X_{\beta}$. 
		\end{itemize}
Since $2^{\lambda}$ has no descending $\lambda^{+}$-sequences, each set $X_{\beta} \cap S_{\beta}$ is nonempty. 
	
	We claim that for all $\beta < \delta$, every member of $S_{\beta}$ is $<_{\rm{lex}}$-below every member of $X_{\delta}$. To see this suppose that $x \in S_{\beta}$ is $<_{\rm{lex}}$-above some $\varphi(\gamma, \delta)$. Then since $x\in S_{\beta}$ and is $<_{\rm{lex}}$-below $\lambda^+$ many members of $X_\beta$, there exists an $\alpha > \gamma$ such that $\varphi(\gamma, \delta)<_{\lex} x<_{\lex}\varphi(\alpha, \beta)$, contradicting \eqref{eq:H}.

	Picking then, for each $\beta \in \lambda^+$, a member of $X_{\beta} \cap S_{\beta}$ gives an increasing sequence of length $\lambda^{+}$ in $K$, which is impossible.
\end{proof}

\section{Acknowledgements}

We thank Trevor Wilson for pointing out the reference \cite{BH}, and Justin Moore for the reference \cite{RinTod}.  We also thank an anonymous referee for their careful reading and helpful comments.


\begin{thebibliography}{99}
\bibitem{BL}
J.~Balogh and A.~Lamaison.
\newblock Ramsey upper density of infinite graph factors.
\newblock {\em arXiv preprint arXiv:2010.13633}, 2020.

\bibitem{BH}
J.~Baumgartner and A.~Hajnal.
\newblock A proof (involving {M}artin's axiom) of a partition relation.
\newblock {\em Fundamenta Mathematicae}, 78(3):193--203, 1973.

\bibitem{B}
A.~Benford.
\newblock {\em On the appearance of oriented trees in tournaments}.
\newblock PhD thesis, University of Birmingham, 2023.

\bibitem{BM1}
A.~Benford and R.~Montgomery.
\newblock Trees with few leaves in tournaments.
\newblock {\em Journal of Combinatorial Theory, Series B}, 155:141--170, 2022.

\bibitem{BM2}
A.~Benford and R.~Montgomery.
\newblock Trees with many leaves in tournaments.
\newblock {\em arXiv preprint arXiv:2207.06384}, 2022.

\bibitem{CDLL}
J.~Corsten, L.~DeBiasio, A.~Lamaison, and R.~Lang.
\newblock Upper density of monochromatic infinite paths.
\newblock {\em Advances in Combinatorics}, page 10810, 2019.

\bibitem{CDM}
J.~Corsten, L.~DeBiasio, and P.~McKenney.
\newblock Density of monochromatic infinite subgraphs II.
\newblock {\em arXiv preprint arXiv:2007.14277}, 2020.

\bibitem{DM}
L.~DeBiasio and P.~McKenney.
\newblock Density of monochromatic infinite subgraphs.
\newblock {\em Combinatorica}, 39(4):847--878, 2019.

\bibitem{DH}
F.~Dross and F.~Havet.
\newblock On the unavoidability of oriented trees.
\newblock {\em Electronic Notes in Theoretical Computer Science}, 346:425--436,
  2019.

\bibitem{E}
P.~Erd\H{o}s.
\newblock Some set-theoretical properties of graphs.
\newblock {\em Compositio Math}, 2:463--470, 1935.


\bibitem{EH1}
P.~Erd\H{o}s and A.~Hajnal.
\newblock Unsolved problems in set theory. 
\newblock In {\em Proc. Sympos. Pure Math}, Vol. 13, no. part 1, pp. 17--48. 1971.


\bibitem{EH2}
P.~Erd\H{o}s and A.~Hajnal.
\newblock Unsolved and solved problems in set theory.
\newblock In: Proceedings of the Tarski Symposium (Proc. Sympos. Pure Math., Vol. XXV, Univ/ California, Berkeley, Calif., 1971), pp. 269-287, Providence, R.I., 1974. Amer. Math. Soc. 

\bibitem{EM}
P.~Erd\H{o}s and J.~W. Moon.
\newblock On sets of consistent arcs in a tournament.
\newblock {\em Canadian Mathematical Bulletin}, 8(3):269--271, 1965.

\bibitem{Fo}
M.~Foreman. 
\newblock Some problems in singular cardinals combinatorics. 
\newblock {\em Notre Dame Journal of Formal Logic}, 46(3):309--322, 2005.



\bibitem{FH}
M.~Foreman and A.~Hajnal.
\newblock A partition relation for successors of large cardinals.
\newblock {\em Mathematische Annalen}, 325(3):583--623, 2003.

\bibitem{F}
R.~Fra{\"\i}ss{\'e}.
\newblock {\em Theory of relations, Revised edition}, volume 145.
\newblock North-Holland, Amsterdam, 2000.

\bibitem{G}
F.~Galvin.
\newblock On a partition theorem of {B}aumgartner and {H}ajnal.
\newblock In {\em Infinite and finite sets ({C}olloq., {K}eszthely, 1973;
  dedicated to {P}. {E}rd\H{o}s on his 60th birthday), {V}ol. {II}}, Colloq.
  Math. Soc. J\'{a}nos Bolyai, Vol. 10, pages 711--729. North-Holland,
  Amsterdam, 1975.

\bibitem{GS}
F.~Galvin and S.~Shelah.
\newblock Some counterexamples in the partition calculus.
\newblock {\em Journal of Combinatorial Theory, Series A}, 15(2):167--174,
  1973.

\bibitem{HT}
F.~Havet and S.~Thomassé.
\newblock Oriented hamiltonian paths in tournaments: A proof of {R}osenfeld's
  conjecture.
\newblock {\em Journal of Combinatorial Theory, Series B}, 78(2):243--273,
  2000.


\bibitem{KMO}
D.~K\"uhn, R.~Mycroft, and D.~Osthus.
\newblock A proof of {S}umner's universal tournament conjecture for large
  tournaments.
\newblock {\em Electronic Notes in Discrete Mathematics}, 38, 2010.

\bibitem{L}
A.~Lamaison.
\newblock Ramsey upper density of infinite graphs.
\newblock {\em arXiv preprint arXiv:2003.06329}, 2020.

\bibitem{LSS}
N.~Linial, M.~Saks, and V.~S\'os.
\newblock Largest digraphs contained in all tournaments.
\newblock {\em Combinatorica}, 3:101--104, 1983.


\bibitem{RinTod}
A.~Rinot, S.~Todorcevic.
\newblock Rectangular square-bracket operation for successor of regular cardinals.
\newblock Fund. Math. 220 (2013), no.2, 119–128.

\bibitem{SS}
M.~Saks and V.~Sós.
\newblock On unavoidable subgraphs of tournaments.
\newblock {\em Proc. 6th Hung. Comb. Coll. Eger, Coll. Math. Soc. J. Bolyai},
  37, 1984.

\end{thebibliography}
\end{document}